\documentclass[11pt]{amsart}
\usepackage[T1]{fontenc}
\usepackage[ansinew]{inputenc}
\usepackage{amssymb,amsmath,amsthm,comment,eucal,hyperref,mathrsfs,array,graphicx,xcolor,bbm,enumerate,wasysym}
\usepackage[foot]{amsaddr}
\usepackage[mathlines,displaymath,pagewise]{lineno}
\usepackage[all]{xy}
\usepackage{fullpage}

\numberwithin{equation}{section}

\newtheorem{lemma}{Lemma}[section]
\newtheorem{propn}[lemma]{Proposition}
\newtheorem{thm}[lemma]{Theorem}

\newtheorem{defn}[lemma]{Definition}
\newtheorem{remark0}[lemma]{Remark}

\DeclareMathOperator{\CR}{CR}

\renewenvironment{proof}{{\em Proof.}}{\hspace*{\fill} $\square$}
\newenvironment{proofof}[1]{{\em Proof of #1.}}{\hspace*{\fill} $\square$}

 \newcommand{\C}{\mathbb{C}}
 \newcommand{\N}{\mathbb{N}}
 
 \newcommand{\I}{\mathbbm{1}}
 \newcommand{\E}[1]{\mathbb{E}\left [ #1 \right ]}

 \newcommand{\Os}{\mathcal{O}}
 \newcommand{\F}{\mathcal F}
 
 \newcommand{\e}{\operatorname{e}}
 
 \newcommand{\D}{\mathbb D }
 \renewcommand{\I}[1]{\mathbf 1_{\{#1\}}}
  \renewcommand{\1}{\mathbf 1}
  \renewcommand{\P}{\mathbb P}

  \newcommand{\leb}{\operatorname{Leb}}

 \newcommand{\eps}{\varepsilon}
 
 \newcommand{\gcn}{n(C,\gamma)}
 \newcommand{\phs}{\hat{\mathbb{P}}^*}
 \newcommand{\ehs}{\hat{\mathbb{E}}^*}

\begin{document}

\title{Critical Liouville measure as a limit of subcritical measures} 
\author{Juhan Aru$^*$ and Ellen Powell$^*$}
\address{$^*$Department of Mathematics, ETH Zürich, Rämistr. 101, 8092 Zürich, Switzerland}
\author{ Avelio Sep\'{u}lveda$^\dagger$}
\address{$^\dagger$Univ Lyon, Université Claude Bernard Lyon 1, CNRS UMR 5208, Institut Camille Jordan, 69622 Villeurbanne, France}
\date{}
\begin{abstract}
	We study how the Gaussian multiplicative chaos (GMC) measures $\mu^\gamma$ corresponding to the 2D Gaussian free field change when $\gamma$ approaches the critical parameter $2$. In particular, we show that as $\gamma\to 2^{-}$, $(2-\gamma)^{-1}\mu^\gamma$ converges in probability to $2\mu'$, where $\mu'$ is the critical GMC measure.
\end{abstract}
\maketitle

\section{Introduction}
Gaussian multiplicative chaos (GMC) theory aims to give a meaning to the heuristic volume form ``$e^\Gamma d\leb$'', where $\Gamma$ is some rough Gaussian field that is not defined pointwise. Such constructions first appeared for Gaussian free fields in the early 70s \cite{HK}, where $\int_D e^\Gamma d\leb$ was defined to  add an exponential interaction to the underlying free field. The theory was then developed for a larger class of Gaussian fields, and named as Gaussian multiplicative chaos, by Kahane \cite{KAH}. 

GMC measures corresponding to the 2D continuum Gaussian free field (GFF), have recently become an active area of study, due to their links to the probabilistic description of 2D Liouville quantum gravity \cite{DS, DKRV}. Out of convention, we will call such measures the Liouville measures \footnote{In the physics literature, ``Liouville measure" refers to a volume form coming from a conformal field theory with a non-zero interaction term (see \cite[Section 3.6]{RVnotes}), inducing a certain weight on the law of the underlying GFF. Therefore, our measures  correspond to a degenerate case, where the interaction parameter is equal to $0$.}.

In this article, we study how the Liouville measures $\mu^\gamma$ vary, for a fixed underlying field, when the parameter $\gamma$ tends to the critical parameter $\gamma = 2$ from below. It is known that the measures $\mu^\gamma$ change analytically in $\gamma$ throughout the subcritical regime $0 \leq \gamma < 2$ (it follows, for example, from a more general result, Theorem 4 of \cite{JJ} \footnote{Indeed, when one considers the approximation of the GFF on $[0,1]^2$ by DGFF-s, then the conditions (4), (7) and the condition on exponential moments can be checked directly; the condition (5) follows from the estimates on the discrete Green's function due to Kenyon, as given for example in Theorem 2.5 of \cite{ChSm}}), and also not hard to show that $\mu^\gamma \to 0$ as $\gamma \to 2$. In Conjecture 9 of \cite{DSRV}, the authors conjecture that $(2-\gamma)^{-1}\mu^\gamma$ converges to a multiple of the so-called critical Liouville measure $\mu'$. The main result of this paper is the confirmation of this conjecture, and determination that the constant is equal to two (see Remark \ref{rem::factor2} for a discussion). More precisely, we prove that:

\begin{thm}\label{thm::conv_derivative}  Let $\Gamma$ be a zero boundary Gaussian free field in a domain $D\subset \C$ and let $\mu^\gamma$ for $\gamma<2$ be the associated sequence of Liouville measures (defined in Theorem \ref{thm::subcritical_mollified}) together with $\mu'$ the critical Liouville measure (defined in Theorem \ref{thm::critical_def}). Then as $\gamma \to 2^{-}$ we have
	\[\frac{\mu^\gamma}{2-\gamma}\to 2\mu'\]
	in probability, with respect to weak convergence of measures. 
\end{thm}
The analogous result, in the setting of multiplicative cascades/branching random walk,  was already known \cite{Madaule}. Our proof strategy is to use the construction of Liouville measure as the multiplicative cascade introduced in \cite{APS} and to then transfer the proof from the case of cascades over to the case of the Liouville measure. Whereas we roughly follow the proof in \cite{Madaule}, in some places there are additional technicalities, and in others there are simplifications. Moreover, we strongly use the results on the Seneta-Heyde scaling of the Liouville measure proved in \cite{APS}.

Our results are a very first step (see Section \ref{sec::extensions}) towards taking $\gamma \to 2^{-}$ limits in the peano-sphere approach to 2D Liouville quantum gravity \cite{DMS}, and also allow one to extend the Fyodorov-Bouchaud formula \cite{Gui} to the critical case.

The rest of the article is structured as follows: we start with basic definitions followed by a  few preliminary lemmas; in Section \ref{sec::proof} we prove the main result; and finally, we discuss some extensions.

\section{Basic definitions} 

Since this article is intended to be a rather brief follow-up to \cite{APS}, we keep the preliminaries to a minimum. We refer the reader to \cite{APS} for more detailed background on the planar Gaussian free field, its local sets and associated chaos measures. 

\subsection{The Gaussian free field and first passage sets}

We denote by $\Gamma$ a Gaussian free field with zero boundary conditions in a simply connected domain $D\subset \C$. 
That is, $\Gamma$ is a centered Gaussian process indexed by the set of smooth functions in $D$, with covariance given by 
\begin{equation}\label{GFF}
\E {(\Gamma,f) (\Gamma,g)}  =  \iint_{D\times D} f(x) G_D(x,y) g(y) d x d y.
\end{equation} 
Here $G_D$ is the Dirichlet Green's function  in $D$. We normalise it so that as $x \to y$,   $G_D(x,y)\sim \log(1/|x-y|)$.

One important characteristic of the Gaussian free field is that it satisfies a spatial Markov property. In fact, it also satisfies a strong spatial Markov property at certain stopping, or ``local'' sets, first studied in \cite{SchSh2}:

\begin{defn}[Local sets]
	Consider a random triple $(\Gamma, A,\Gamma_A)$, where $\Gamma$ is a  GFF in $D$, $A$ is a random closed subset of $\overline D$  and $\Gamma_A$ a random distribution that can be viewed as a harmonic function, $h_A$, when restricted to $D \setminus A$.
	We say that $A$ is a local set for $\Gamma$ if conditionally on $A$ and $\Gamma_A$, $\Gamma^A:=\Gamma - \Gamma_A$ is a (zero-boundary) GFF in $D \setminus A$.
\end{defn}

One particularly nice class of local sets are those corresponding to the first hitting time of level $a\geq 0$ of a Brownian motion. They are called first passage sets (FPS), were introduced in \cite{ALS}, and are characterised by the following proposition:
\begin{propn}[First passage sets]
	Let $a\geq 0$ and $\Gamma$ be a GFF in $D$. Then, there exists a unique local set $A_{a}$ of $\Gamma$ such that:
	\begin{enumerate}[i)]
		\item $h_{A_{a}}=a$
		\item $a-\Gamma_{A_{a}}$ is a positive measure.
	\end{enumerate}
	$A_{a}$ is called the ``first passage set'' of level $a$ of $\Gamma$.
\end{propn}

We will need the following simple properties of the FPS, see e.g. \cite{APS} for explanations and \cite{ALS,ALS2} for further properties.
\begin{propn} Let $a,\delta\geq 0$ and $\Gamma$ be a GFF, then
	\begin{enumerate}
		\item $A_{a}$ has 0 Lebesgue measure.
		\item $A_{a+\delta}$ can be explored by first exploring $A_{a}$, and then inside every connected component $O$ of $D\backslash A_{a}$, exploring $A_{\delta}$ of $\Gamma^{A_{a}}$ restricted to $O$. In particular $A_a\subseteq A_{a+\delta}$.
	\end{enumerate}
\end{propn}

\subsection{Construction of the Liouville measures}

We will now briefly define the ``Liouville'' measures: that is, the family of multiplicative chaos measures corresponding to the 2D Gaussian free field. 

\subsubsection{Subcritical regime} 
When $\gamma<2$, the construction and properties of these measures 
are now rather well understood (see, for example, reviews \cite{Ber,RVnotes,Anotes}). The standard construction goes as follows. Let $\rho:\C \to [0,\infty)$ be a smooth function of unit mass, supported on the unit disc, and for $z\in D$ and $\eps>0$, set $\rho_z^\eps(w):=\eps^{-2}\rho(\eps^{-1}(w-z))$, so that $\Gamma_\eps(z):=(\Gamma, \rho_z^\eps)$ is an approximation to $\Gamma$. Define the corresponding sequence of approximate measures
\begin{equation}\label{eqn::moll_approx} \mu_\eps^\gamma(dz):= \e^{\gamma \Gamma_\eps(z)}\eps^{\gamma^2/2} \, dz.\end{equation}
The chaos measure $\mu^\gamma$ is then defined by taking a limit of these measures as $\eps\to 0$. 
\begin{thm}\label{thm::subcritical_mollified}{\cite{DS,RV,Ber}}
	For $\gamma < 2$ the measures $\mu^\gamma_\eps$ converge to a non-trivial measure $\mu^\gamma$ weakly in probability. Moreover, for any fixed Borel set $\Os \subseteq D$ we have that $\mu^\gamma_\eps(\Os)$ converges in $L^1$ to $\mu^\gamma(\Os)$. 
\end{thm}
This measure is also unique, in the sense that the same limit is obtained if one replaces $\Gamma_\eps$ with any (nice enough) mollifier approximation to $\Gamma$.

In \cite{APS}, it was shown that $\mu^\gamma$ for $\gamma<2$ can alternatively be constructed using the local sets of $\Gamma$. In this paper, we will use the explicit construction using first passage sets: for $n\in \N$ and $\gamma\ge 0$, we define the measures
\begin{equation}\label{eqn::mn_def} M_n^\gamma(dz):= \e^{\gamma n}\CR(z,D\setminus A_n)^{\gamma^2/2} \, dz\end{equation}
where $A_n$ is the $n$-FPS of $\Gamma$, and for $z\in D$, $\CR(z,D\setminus A_n)$ is the conformal radius seen from $z$ of the connected component of $D\setminus A_n$ containing $z$. 
\begin{thm}{\cite[Proposition 4.1]{APS}} \label{thm::subcritical_cascades} For $\gamma<2$, almost surely as $n\to \infty$, $M_n^\gamma$ converges  to $\mu^\gamma$ with respect to the weak-topology of measures. Moreover, for any $\Os\subset D$, $M_n^\gamma(\Os)$ is a martingale that converges almost surely and in $\mathcal{L}^1$ to $\mu^\gamma(\Os)$. Furthermore, the law of $\mu^\gamma(\Os)$ given $A_n$ is that of 
	\begin{equation}\label{decomponsition}
	e^{\gamma n}\sum_{D'\in \mathcal{A}_n} \tilde{\mu}_{D'}^\gamma(\Os\cap D'),
	\end{equation}
	where $\mathcal A_n$ is the set of connected components of $D\backslash A_n$ and $(\tilde{\mu}_{D'})_{D'\in \mathcal A_n}$ is a sequence of (conditionally) independent Liouville measures in $(D')_{D'\in \mathcal{A}_n}$.
\end{thm}

Let us remark that such a simple construction of the Liouville measure was first proposed (but not proved) in \cite{Ai}.

\subsubsection{Critical regime} For $\gamma\ge 2$ it is known, \cite{RV,APS}, that the approximate measures \eqref{eqn::moll_approx} and \eqref{eqn::mn_def} converge to the zero measure almost surely. Thus, to obtain a non-trivial limit in these cases one must renormalise differently. We concern ourselves here only with the critical case $\gamma=2$. 

In this case there are two procedures that one can use to obtain a non-trivial limit \cite{DSRV,DSRV2}. The first is a deterministic renormalisation, known as the Seneta--Heyde rescaling, where one considers the sequence of measures 
either 
\begin{equation}
\sqrt{\log(1/\eps)} \mu_\eps^{\gamma=2} \quad \text{or} \quad \sqrt{n}M_n^{\gamma=2} 
\end{equation}
depending on whether you want to approximate using mollifiers or local sets. The second is a random ``derivative'' normalisation, where one considers the sequence of measures either
\begin{eqnarray}
D_\eps(dz)& := & (-\Gamma_\eps(z)+2\log (1/\eps))\e^{2 \Gamma_\eps(z)}\eps^{2}dz  \quad \text{or} \\ \nonumber  D_n(dz) & := & (-n+2\log \CR^{-1}(z,D\setminus A_n)) \e^{2n}\CR(z,D\setminus A_n)^{2} dz.
\end{eqnarray}

It is now known that all the above approximations converge to the same (up to a constant) limiting measure. The following is a combination of results of \cite{DSRV, DSRV2, HRVdisk, JS, EP}: 
\begin{thm}\label{thm::critical_def}
	The sequence of measures $D_\eps$ converge weakly in probability as $\eps\to 0$ to a limiting measure $\mu'$. Furthermore, $\sqrt{\log(1/\eps)}\mu_\eps^{\gamma=2}$ converges weakly in probability to $\sqrt{\frac{2}{\pi}} \mu'$.
\end{thm}

\begin{thm}{\cite[Proposition 6.4 and Theorem 6.6]{APS}}\label{thm::critical_cascades}
	The sequence of measures $D_n$ converge weakly almost surely to $\mu'$ as $n\to \infty$. Furthermore, the sequence of measures $\sqrt{n}M_n^{\gamma=2}$ converge weakly in probability to $\frac{2}{\sqrt{\pi}}\mu'$. 
\end{thm}	

\subsection{Rooted measures} \label{sec::com}

One of the key techniques used to study chaos measures, is to work with certain ``rooted'' probability measures. This idea goes back to Peyri\`ere \cite{KP}. It has been widely used in the classical ``spine'' theory of branching processes \cite{BK}, as well as to study the law of the field plus a ``typical'' point under the Liouville measure \cite{DS}. We will introduce them in the setting of FPS, as employed in \cite{APS}.

We define, for $\gamma\ge 0$, a probability measure on the field $\Gamma$ plus a distinguished point $Z$, by
\begin{equation} \label{eqn::def_rooted_meas}
\left.\hat{\mathbb{P}}^*_\gamma(d\Gamma, dz)\right|_{\mathcal{F}_{n}^*}:= \frac{\e^{\gamma n}\CR(z,D\setminus A_n)^{\gamma^2/2}}{\int_D \CR(x,D)^{\gamma^2/2} \, dx} \mathbb{P}(d\Gamma) \, dz,
\end{equation}
where $\mathcal{F}_{n}^*:=\sigma(Z)\vee\sigma(A_n)$, and set $\hat{\mathbb{P}}^*:=\hat{\mathbb{P}}^*_2$ (this is the measure we will work with most often). We make the following straightforward observations concerning the law of the random variables $\Gamma$ and $Z$ under $\hat{\mathbb{P}}^*_\gamma$:
\begin{itemize}
	\item The marginal law of the field $\Gamma$  restricted to $\sigma(A_n)$ is given by $(M_0^\gamma(D))^{-1}M_n^\gamma(D) \mathbb{P}(d\Gamma)$ and is absolutely continuous w.r.t the law of the GFF.
	\item The marginal law of the point $Z$ has density (w.r.t Lebesgue measure) $\propto \CR(z,D)^{\gamma^2/2}$.
	\item Conditionally on $A_n$, the point $Z$ is chosen proportionally to $M_n^\gamma(dz)$.
\end{itemize}

Slightly less immediate is the following description of the conditional law of the field given the point $Z$, see for example \cite[Lemma 2.1]{Anotes} for a proof of the following statement, which concerns the subcritical regime:
\begin{lemma}\label{lem::rooted_measure}
	For $\gamma\in (0,2)$, one can sample from $\hat{\mathbb{P}}^*_\gamma$ by first sampling the point $Z$ proportionally to $\CR(Z,D)^{\gamma^2/2}$, and then sampling the field according to the law of $\Gamma+ \gamma G_D(Z,\cdot)$. We write $\hat{\mathbb{P}}^*_{\gamma,z}$ for the law of $\Gamma+\gamma G_D(z,\cdot)$.	
\end{lemma} 
This lemma tells us roughly that, for $\gamma<2$, the GFF around a typical point sampled from the Liouville measure has an additional $\gamma$ singularity. In the case of the FPS approximation, this is encoded in a certain random walk. More precisely, if under the conditional law $\hat{\mathbb{P}}_\gamma^*(\cdot | Z)$ we set
\begin{equation}
S^\gamma_n= S^\gamma_n(Z):= -\gamma n+\gamma^2 \log\CR^{-1}(Z,D\setminus A_n)
\end{equation}
then $S^\gamma_n(Z)-S^\gamma_0(Z)$ is a centred random walk, the law of whose increments do not depend on the point $Z$, and have mean zero and variance $\gamma$ (see Remark 6.3 of \cite{APS}). 

In fact for us, the most important case is when $\gamma = 2$, which is not included in Lemma \ref{lem::rooted_measure}. However, we still know (see \cite[Section 6]{APS}) that under $\hat{\mathbb{P}}^*=\hat{\mathbb{P}}_2$ and conditionally on the point $Z$, the random walk $S_n := S^{2}_n$ is as described in the previous paragraph, with $\gamma=2$. 

\section{Preliminary lemmas}

Here, we collect a few slightly technical preliminary lemmas. One may safely skip this section in the first reading, and only return to them when they appear in the main proofs.

\subsection{Uniform control of Liouville moments} 

We will need some control on the moments of the subcritical Liouville measures. In the following we write $\mu_D^\gamma$ for the $\gamma$-Liouville measure defined from a zero boundary Gaussian free field on $D$ as in Theorem \ref{thm::subcritical_mollified}.
\begin{lemma}\label{lem::pmoment}
	Let $p=p(\gamma):=1+\frac{2-\gamma}{2}$. There exists a universal $K>0$, such that for any $f:\C\to [0,1]$, any $\gamma\in (1,2)$, and any simply connected domain $D\subseteq \D$, if $\mu^\gamma_D$ is the limiting measure associated with a zero boundary GFF $\Gamma$ on $D$:
	\[ \mathbb{E}\left [\left (\int_D f(z) \mu^\gamma_D(dz)\right )^p\right ] \leq K \mathrm {Area}(D)^{(p-1)(1+\gamma^2/4)} \int_D f(z)^p \CR(z,D)^{\gamma^2/2} dz.\] 
\end{lemma}
\vskip3mm
The key ingredient is a uniform control on $(p-1)$-th moments under the rooted measure, with $p$ as above.
\begin{lemma} \label{lem::centre_z_pmoment}
	There exists a universal constant $C > 0$ such that for all $\gamma\in (0,2)$ and $p=p(\gamma):=1+\frac{2-\gamma}{2}<2$ we have (recalling the definition of $\hat{ \mathbb E}_{\gamma,z}$ from Lemma \ref{lem::rooted_measure})
	\[
	\hat{\mathbb{E}}^*_{\gamma,0}[\mu_\D^\gamma(\D)^{p-1}]\le C.\]
\end{lemma}

We will first show Lemma \ref{lem::centre_z_pmoment}.
\vskip 0.005\textheight
\begin{proofof}{Lemma \ref{lem::centre_z_pmoment}}
	First, by Lemma \ref{lem::rooted_measure}, we can write
	\begin{equation}\label{eq:: tilted LQG measure}
	\ehs_{\gamma,0}[\mu^\gamma_{\D}(\D)^{(p-1)}] = \E{\left(\int_{\D}\frac{1}{|z|^{\gamma^2}}d\mu^\gamma_{\D}\right)^{p-1}}.
	\end{equation}
	
	Consider the radial decomposition of $\Gamma$, i.e. write (in the sense of distributions) $\Gamma(z)$ as $B_{|z|} + \Gamma^\sphericalangle(z)$, where $B_r$ has the law of a standard Brownian motion when parametrized by $-\log r$ and $\Gamma^\sphericalangle$ is a log-correlated Gaussian field, whose circle-averages around the origin are zero (see e.g. \cite{DMS}). Writing \[d\mu^\gamma_{\Gamma^\sphericalangle}(z):=\lim_{\eps\to 0} \e^{\Gamma^\sphericalangle_\eps(z)}\eps^{\frac{\gamma^2}{2}}\e^{\frac{\gamma^2}{2} \text{var}(B_{|z|})} dz\] for the angular GMC measure, we have 
	\[\E{\left(\int_{\D}\frac{1}{|z|^{\gamma^2}}d\mu_\D^\gamma\right)^{p-1}} = \E{\left(\int_{\D}\frac{1}{|z|^{\gamma^2/2}}e^{\gamma B_{|z|}}d\mu^\gamma_{\Gamma^\sphericalangle}\right)^{p-1}}.\]
	By first conditioning on $B_{|z|}$, using the fact that $\E{ d\mu_{\Gamma^\sphericalangle}^\gamma(z)} = \CR(z,\D)^{\gamma^2/2}dz$ and Jensen's inequality, we can bound the RHS further by a constant times
	\[\E{\left(\int_{\D}\frac{1}{|z|^{\gamma^2/2}}e^{\gamma B_{|z|}}dz\right)^{p-1}}.\]
	Now for $n \in \N $, consider the decomposition of the radial part into intervals of the form 
	\[[r_{n+1}, r_n] = \left [2^{-(n+1)(p-1)^{-2}},2^{-n(p-1)^{-2}}\right ].\] Denote by $R_n$ the corresponding annulus, and observe that:
	\begin{itemize}
		\item $B_{|r_n|}$ is a zero-mean Gaussian of variance $-\log |r_n|$;
		\item the maximum of the process $B_s - B_{r_n}$ (that is a time-changed Brownian motion) over the interval $s \in [r_{n+1},r_n]$ is a sub-Gaussian of mean bounded by an absolute constant times $(p-1)^{-1}$ and of variance bounded by an absolute constant times $(p-1)^{-2}$;
		\item $(2-\gamma)^{-(p-1)} \to 1$.
	\end{itemize}
	Thus, there exist a universal constant $K>0$ such that
	\[\E{\left(\int_{R_n}\frac{1}{|z|^{\gamma^2/2}}e^{\gamma B_{|z|}}dz\right)^{p-1}} \leq K r_n^{(p-1)(2-\gamma^2+\gamma^3/4)}.\]
	Note that $2-\gamma^2+\gamma^3/4$ is strictly positive for $\gamma\in [0,2)$. 
	Thus by the sub-additivity of $x^{p-1}$, we have that
	\[\E{\left(\int_{\D}\frac{1}{|z|^{\gamma^2/2}}e^{\gamma B_{|z|}}dz\right)^{p-1}} \leq K\sum_{n \geq 0} r_n^{(p-1)(2-\gamma^2+\gamma^3/4)} = \frac{K}{1-2^{-(2-\gamma^2+\gamma^3/4)(p-1)^{-1}}}.\]
	Finally, as $(2-\gamma^2+\gamma^3/4)(p-1)^{-1}= \gamma+2-\gamma^2/2>2$ for all $\gamma\in[0,2]$, we can conclude.
\end{proofof}

Let us now prove Lemma \ref{lem::pmoment}.
\vskip 0.005 \textheight

\begin{proofof}{Lemma \ref{lem::pmoment}} As $p < 4/\gamma^2$, standard theory of GMC \cite{KAH} guarantees that $\mathbb{E}[(\mu_D^\gamma(D))^p]$ is finite for any $\gamma \in (0,2)$. Now, by Jensen's inequality we have
	\begin{equation}\label{eq:: pmoment}\mathbb{E}\left[\left (\int_D f(z) \mu_D^\gamma(dz)\right )^p\right]\le \mathbb{E}\left[( \mu_D^\gamma(D))^{p} \int_D f(z)^p \frac{\mu^\gamma_D(dz)}{\mu_{D}^\gamma(D)}
	\right],\end{equation} which by definition of the measure $\hat{\mathbb{E}}_\gamma^*$ is also equal to $\int_D CR(w,D)^{\gamma^2/2}dw$ times 
	\[	\hat{\mathbb{E}}^*_\gamma\left[ \int_D f(z)^p \frac{\mu^\gamma_D(dz)}{\mu^\gamma_D(D)} \times \mu_{D}^\gamma(D)^{p-1}\right]=\hat{\mathbb{E}}^*_\gamma \left[ f(Z)^p\mu^\gamma_D(D)^{p-1}\right]. \]
	Since the marginal density of $Z$ is proportional to $\CR(Z,D)^{\gamma^2/2}$, by conditioning on $Z$ the LHS of \eqref{eq:: pmoment} can be bounded by
	\[
	\int_D \CR(z,D)^{\gamma^2/2} f(z)^p\hat{\mathbb{E}}^*_{\gamma,z}\left[ \mu^\gamma_{D}(D)^{p-1}\right]dz.\]
	
	Now, let $B^z$ be the ball of radius $\CR(z,D)/8$ around $z$ so that $B^z\subset D$. Then, we have that 
	\begin{equation}\label{eqn::splitD} 	
	\hat{\mathbb{E}}^*_{\gamma,z}[\mu^\gamma_{D}(D)^{p-1}] \le \hat{\mathbb{E}}^*_{\gamma,z}[\mu^\gamma_{D}(B^z)^{p-1}] +\hat{\mathbb{E}}^*_{\gamma,z} [\mu_{D}^\gamma(D\setminus B^z)^{p-1}].\end{equation} Since for $y\in D\setminus B^z$, $G_D(z,y)$ is bounded above by some universal constant (by conformal invariance of the Green's function and the distortion theorem \cite[Theorem 3.23]{Lawler}), the second term in \eqref{eqn::splitD} can be bounded by a universal constant times
	\[ \E{\mu_{D}^\gamma(D)^{p-1}} \leq \text{Area}(D)^{p-1}\sup_{y\in D} \CR(y,D)^{\frac{\gamma^2}{2}(p-1)} \le 10^{\frac{\gamma^2}{4}(p-1)}\text{Area}(D)^{(1+\frac{\gamma^2}{4})(p-1)},\] where in the second inequality we use that $\CR(z,D)^2 \leq 10 \text{Area}(D)$.
	
	For the first, we apply the scaling and translation map $\phi:w\mapsto 8(w-z)/\CR(z,D)$, which sends $B^z$ to $\D$ and, by the Koebe $1/4$-Theorem, $D$ to $\phi(D)\supset 2\D$. Letting $\mu_{\phi(D)}^\gamma$ denote the Liouville measure associated to a GFF in $\phi(D)$, by scaling invariance of the GFF and the Green's function we thus see that the first term of \eqref{eqn::splitD} is less than or equal to a universal constant times
	\[
	\CR(z,D)^{(2+\gamma^2/2)(p-1)}\hat{\mathbb{E}}^*_{\gamma,0}[\mu_{\phi(D)}^\gamma(\D)^{p-1}].\]
	This, using again that $\CR(z,D)^2 \leq 10 \text{Area}(D)$ and Kahane's convexity inequality \cite{KAH}, can be bounded by a constant times
	\[\text{Area}(D)^{(1+\frac{\gamma^2}{4})(p-1)}\hat{\mathbb{E}}^*_{\gamma,0} [\mu^\gamma_{\D}\left (\D\right )^{p-1}],\]
	and the claim now follows from Lemma \ref{lem::centre_z_pmoment}. 
\end{proofof}

\subsection{A random walk associated to the root}

Recall from Section \ref{sec::com}, that under the conditional law $\hat{\mathbb{P}}^*(\cdot | Z)$ if we set 
\begin{equation}
\label{eqn::def_sn}
S_n= S_n(Z):= -2n+4\log\CR^{-1}(Z,D\setminus A_n)
\end{equation}
then $S_n(Z)-S_0(Z)$ is a simple random walk, whose distribution does not actually depend on the point $Z$, and whose increments have mean zero and variance equal to 2 \cite[Lemma 6.1, Remark 6.3]{APS}. Additionally note that $S_{n+1}-S_n$ is always greater than $-2$, because the conformal radius is strictly decreasing in $n$. One can also extract easily from this proof that $(S_n-S_0)$ has exponential moments. 

Write $E_\eta(n,z)=\{S_k(z)\ge -2\eta\, , \, 0\le k \le n\}$. We will later need the following lemma controlling the exponential moments of the conditioned walk:

\begin{lemma}\label{lem::control_exp_big}
	Fix $C>1$ and a deterministic sequence $C_n\to C$. Then there exists $c(C)>0$ and $n_0(C)>0$ such that for any $p,n\ge n_0$ 
	\begin{equation}
	\hat{ \mathbb E}^*\left(\left. \e^{\frac{C_nS_n(Z)}{2\sqrt{n}}}\I{\frac{S_n(Z)}{\sqrt{n}}\ge p} \1_{E_\eta(n,Z)}\right| Z\right)\le c\frac{1+(2\eta+S_0(Z)/\sqrt{n})^+}{\sqrt{n}} \e^{-\frac{p}{4}}\e^{\frac{CS_0(Z)}{\sqrt{n}}}
	\end{equation}
\end{lemma}
This lemma is a direct consequence of a more general lemma, by checking that all the conditions hold for our random walk $S_n$. This is the analogue of \cite[Lemma A.2]{Madaule}, but we give a different proof.
\begin{lemma}
	Let $X_n=\sum_{k=1}^n y_k$ be a random walk with increments of mean 0 and variance $\sigma^2$ such that $\P(y_k<-1)=0$ and $\E{ e^{\epsilon y_1} }<\infty$ for some $\epsilon>0$. Then, for all $C$ and $C_n\to C$, there exists $c=c(C)$ and $n_0(C)>0$ such that for any $n>n_0,p>1,a>0$
	\begin{equation}\label{A.2}
	\E{e^{\frac{C_nX_n}{\sqrt{n}}}\I{\frac{X_n}{\sqrt{n}}\geq p}\I{\inf_{k\leq n} X_k\geq -a} }\leq (a+1)c \frac{e^{-p/4}}{\sqrt n} 
	\end{equation}
\end{lemma}

\begin{proof}
	First note that for any $m\leq n$ and $n\geq (C_n+1)/\epsilon$  we have that
	\begin{linenomath}	
		\begin{align}\label{eqn::walk1}
		\E{e^{\frac{C_nX_m}{\sqrt{n}}}\I{\frac{X_m}{\sqrt{n}}\geq p} } &\leq e^{-p} \E{e^{\frac{(C_n+1)X_m}{\sqrt{n}}}}=e^{-p}\left( \E{  \exp\left (\frac{(C_n+1)}{\sqrt{n}}X_1\right)} \right)^m\nonumber\\ 
		&=\exp\left(\frac{m}{n} ((c+1)^2\sigma^2/2+o(n^{-1/2}))-p\right)<\tilde c(C) e^{-p},
		\end{align}
	\end{linenomath}
	where we have used analyticity of the Laplace transform near $0$ in the second line.
	Now define $\tau$ to be the first time $X_n$ exits $[-a,\sqrt{n}p/(2(C_n+1))]$. We can use the strong Markov property and the fact that $X_{\tau\wedge n}\geq -(a+1)$, to conclude that
	\begin{linenomath}
		\begin{align*}
		(a+1)\geq \E{X_{\tau\wedge n} \I{X_{\tau}>\sqrt{n}p/(2(C_n+1)),\tau\leq n}}\geq \frac{\sqrt{n}p}{2(C_n+1)}\P( X_{\tau}>\sqrt{n}p/(2(C_n+1)),\tau<n)
		\end{align*}
	\end{linenomath}
	and thus, 
	\begin{equation}\label{eqn::walk2}
	\P( X_{\tau}>\sqrt{n}p/2(C_n+1),\tau<n) = p^{-1}(a+1)O(1/\sqrt n).
	\end{equation} 
	Using \eqref{eqn::walk1} we can bound the left hand-side of \eqref{A.2} by
	\begin{linenomath}
		\begin{align*}
		&\E{e^{\frac{C_nX_{\tau}}{\sqrt{n}}}\I{X_\tau>-a,\tau\leq n}\E{e^{\frac{C_n(X_n-X_\tau)}{\sqrt{n}}}\I{\frac{X_m-X_\tau}{\sqrt{n}}\geq p-\frac{X_\tau}{\sqrt{n}}}  \mid \F_\tau}}\leq \tilde c e^{-p} \E{e^{\frac{(C_n+1)X_\tau}{\sqrt{n}}}\I{X_\tau>-a,\tau\leq n} }.
		\end{align*}
	\end{linenomath}
	
	Finally, observe that $(C_n+1)X_{\tau-1}/\sqrt{n}$ is smaller than or equal to $p/2$. Thus by separating the cases of $y_\tau$ greater or less than $ \sqrt{n}p/4(C_n+1)$, we can further bound the right hand side by
	\[ce^{-p/4}\P(X_\tau>\sqrt{np}/2,\tau\leq n)+\tilde ce^{-p/2}\E{e^{\frac{(C_n+1)y_\tau}{\sqrt{n}}} \I{y_\tau\geq \sqrt{n}p/4(C_n+1) }}\]
	By \eqref{eqn::walk2}, the first term is bounded by $(a+1)ce^{-p/4}(\sqrt{n}p)^{-1}$. Moreover, by Cauchy-Schwarz and the assumption of exponential moments, the second term is smaller than $c e^{-p/2}ne^{-\epsilon \sqrt{n}p/2(C_n+1)}$. Thus, the lemma follows.
	
\end{proof}

\subsection{First passage set seen from the root}

For a later technical argument we will also need to have some control on the geometry of first passage sets with respect to the marked point $Z$ under the rooted measures $\hat{ \mathbb P}^*_\gamma$. The following lemma comes from \cite{Ai}. 

\begin{lemma}\label{lem::ZbigCR} Set $D=\D$. There exist $c,c'>0$ such that if $D_n(Z)$ is the component of $\D \setminus A_n$ containing $Z$, we have 
	\[ \phs_\gamma\left(\frac{\text{Area}(D_n(Z))}{\CR(Z,D_n(Z))^{2}}\ge K\right) \leq ce^{-c'n}+K^{-c'} \]
	for all $\gamma\in (1,2)$ and $K\ge 0$.
\end{lemma}

\begin{proof}
	This statement is part of \cite[Lemma 2.3(iii)]{Ai}. Uniformity of $c,c'$ in $\gamma\in (1,2)$ is not explicitly stated in this lemma, but it comes directly from the proof.
\end{proof}

\section{Proof of Theorem \ref{thm::conv_derivative}}\label{sec::proof}

By conformal invariance of the Gaussian free field, we may take $D=\D$. As mentioned in the introduction, the proof follows closely the strategy in \cite[Proof of Theorem 1.1]{Madaule}. However, the presentation is self-contained and some technical details differ. 

By a standard argument (e.g. see \cite[Remark 4.3]{APS}), Theorem \ref{thm::conv_derivative} follows, once we show that for every $\Os\subseteq \D$
\[\frac{\mu^{\gamma}(\Os)}{2-\gamma}\to 2\mu'(\Os)\] in probability, as $\gamma\nearrow 2$.

This in turn follows from a diagonal argument: we will define below an approximation speed $n(C,\gamma)$ such that on the one hand the level $n$ approximations of $\mu^\gamma$ converge to $2\mu'$, and on the other hand the error of the approximations w.r.t $\mu^\gamma$ go to zero. These steps are separated into two lemmas:

\begin{lemma} \label{lem::main1}
	For any $\Os\subseteq \D$ 
	\[	\lim_{C\to \infty} \lim_{\gamma\to 2^{-}}\frac{M_{\gcn}^\gamma(\Os)}{2-\gamma} = 2\mu'(\Os)\]
	where the limit is in $\mathbb{P}$-probability. 
\end{lemma}

\begin{lemma}\label{lem::main2}
	For any $\Os \subseteq \D$ and $\eps>0$
	\[\limsup_{C\to \infty} \limsup_{\gamma\to 2^{-}} \mathbb{P} \left( \left|\frac{M^{\gamma}_{\gcn}(\Os) - \mu^{\gamma}(\Os)}{2-\gamma}\right|>\eps\right) = 0\]
\end{lemma}
It turns out that the right choice of $n(C,\gamma)$ is given by
\begin{equation}
\label{eqn::ncgamma}
n(C,\gamma) := \left\lfloor \left(\frac{C}{2-\gamma}\right)^2 \right\rfloor,\end{equation}
and that it is also necessary to include the dependence on the extra parameter $C$. Before going to the proof, let us try to briefly discuss this choice.

Let us first consider $(2-\gamma)^{-1}M_{n}^\gamma(\D)$ to see which choices of approximation level $n=n(\gamma)$ could possibly give us the right limit. Notice that for $\gamma<2$ we can write 
\[\frac{M_n^\gamma(\D)}{2-\gamma} = \frac{1}{2-\gamma}\int_\D \e^{(\gamma-2)n}\CR(z,\D\setminus A_n)^{\frac{\gamma^2}{2}-2} M_n^{2}(dz) \]
Now, we know from Theorem \ref{thm::critical_cascades} that we have to multiply the measures $M_n^{2}(dz)$ by $\sqrt{n}$ in order to converge to a multiple of the critical measure. Thus, forgetting about the first terms in the integrand, it seems that in order to obtain a non-trivial limit we should pick $n(\gamma)\propto (2-\gamma)^{-2}$. 

So, let us consider $n(\gamma)=n(C,\gamma)=\lfloor (C/(2-\gamma))^2 \rfloor$ for some $C>0$. In the proof of Lemma \ref{lem::main1}, we will see that as $\gamma\to 2$ the measures $(2-\gamma)^{-1}M_n^\gamma$ converge to $c_1(C)\times \mu'$ for some $C-$dependent constant $c_1(C)$. This hints that for any fixed $C$, the error introduced when approximating $(2-\gamma)^{-1} \mu^\gamma$ by $(2-\gamma)^{-1} M_{n}^\gamma$ with $n=n(\gamma,C)$ does not go to $0$ as $\gamma\to 2$. Taking the extra limit $C \to \infty$ allows us to  control this error.

\begin{remark0}\label{rem::factor2}
	
	Finally, let us comment on the slightly surprising factor of $2$. Essentially the reason is the same as that given in \cite[below Theorem 1.1]{Madaule}, but we explain it here in our context.
	
	We saw in Section \ref{sec::com} that for a typical point $z=Z$ sampled from the measure $\mu^\gamma$ 
	\begin{equation}
	S^\gamma_n(z)= -\gamma n+\gamma^2 \log\CR^{-1}(z,D\setminus A_n)
	\end{equation}
	is a mean zero random walk. Now for any $\gamma \in [0,2]$, we decompose the GMC measure according to the sign of the walk
	\[M_n^\gamma(dz) := M_n^\gamma(dz)\I{S^\gamma_n(z) \geq 0} + M_n^\gamma(dz)\I{S^\gamma_n(z) \leq 0},\]
	and denote the summands respectively by $M_n^{\gamma,+}(dz)$ and $M_n^{\gamma,-}(dz)$. Decompose similarly the derivative martingale
	\[D_n^{\pm}(dz) = -\partial_\gamma|_{\gamma = 2}M_n^{\gamma,\pm}(dz).\] The origin of the factor $2$ can now be explained by the following observations:
	\begin{itemize}
		\item For $\gamma < 2$, both $M_n^{\gamma,+}$ and $M_n^{\gamma,+}$ converge to some non-trivial measures $M^{\gamma,+}$ and $M^{\gamma,-}$as $n \to \infty$. Moreover $M^{\gamma,+}+M^{\gamma,-}=\mu^{\gamma}$.
		\item For the derivative martingale, however, $\lim_{n \to \infty}D_n^{-} = 0$, whereas $\lim_{n \to \infty}D_n^{+} = \mu'$ (see \cite{APS} for an explanation; for the same reason that $M_n^{2}\to 0$ as $n\to \infty$ we see that the limit of $D_n$ is supported only where $S_n$ is large). Thus, $\mu'$ is only the limit of the derivatives of $M_n^+$ as $n \to \infty$.
		\item Finally,  $\mu'$ is the limit of both the derivatives of $M^{\gamma,+}$ and $M^{\gamma,-}$, i.e.
		\[\lim_{\gamma \to 2^-}(2-\gamma)^{-1}M^{\gamma,+} = \lim_{\gamma \to 2^-}(2-\gamma)^{-1}M^{\gamma,-}= \mu'.\]
		Indeed, this follows from a direct calculation, after observing that in the proof of Lemma \ref{lem::main1} the term $\E{\e^{\frac{C}{\sqrt{2}}R_1}}$ will be replaced by \[\E{\e^{\frac{C}{\sqrt{2}}R_1}\I{C/\sqrt{2} \le R_1}} \text{ and by } \E{\e^{\frac{C}{\sqrt{2}}R_1}\I{R_1\le C/\sqrt{2}}},\] when considering $M^{\gamma,+}$ or $M^{\gamma,-}$, respectively.
	\end{itemize}
	In other words, the factor $2$ originates from the fact that taking the limit as $n\to \infty$ and taking the derivative in $\gamma$ do not commute: when one first takes the derivative and then the limit $n \to \infty$, the contribution of $\partial_{\gamma}\mid_{\gamma=2}M^{\gamma,-}$ disappears.
\end{remark0}

\medskip

\begin{proofof}{Lemma \ref{lem::main1}}
	From now on, we work for simplicity in the case $\Os=\D$ and also write $n=\gcn$ to try and keep notations compact. 
	
	Our main input is \cite[Theorem 6.7]{APS}, which says that for any positive, continuous bounded function $F$ on $\D$ 
	\begin{equation}\label{eqn::extended_SH} \frac{\sqrt{n}}{D_n(\D)} \int_{\D} \e^{2n-2l(z,n)} F\left(\frac{S_n(z)}{\sqrt{n}}\right) \, dz \to \sqrt{\frac{4}{\pi}} \E{F(\sqrt{2}R_1)}. 
	\end{equation}
	as $n \to \infty$ in probability. Here $l(z,n):=\log \CR^{-1}(z,\D\setminus A_n)$, $S_n(z):=-2n+4l(z,n)$ and $R_1$ has the law of a Brownian meander at time 1.
	
	We will aim to write $(2-\gamma)^{-1} M_{n}^\gamma$ in a similar form. 
	First, by the definition of $n = n(C,\gamma)$ in \eqref{eqn::ncgamma}, we can write 
	\begin{equation*}
	\frac{M_n^\gamma(\D)}{2-\gamma} = D_n(\D) \times \frac{1}{C} \times \frac{\sqrt{n}(1+o(1))}{D_{n}(\D)} \int_{\D} \e^{2n-2l(z,n)} \e^{(\gamma-2)n - (\frac{\gamma^2}{2}-2)l(z,n)}\, dz	\end{equation*}
	where by $o(1)$ we mean a deterministic function of $\gamma$ (possibly depending on $C$) that converges to $0$ as $\gamma\to 2$ and that comes from the fact that $C^2(2-\gamma)^{-2}$ may not be an integer. Substituting now $S_n(z):=-2n+4l(z,n)$, we further rewrite this as
	\begin{equation}\label{eqn::expansion_mainlem1}
	(1+o(1)) \times D_n(\D) \times \frac{e^{-\frac{C^2}{4}}}{C} \times \frac{\sqrt{n}}{D_n(\D)} \int_{\D} \e^{2n-2l(z,n)}\e^{\frac{C}{2}\frac{S_n(z)}{\sqrt{n}}(1+o(1))} \, dz.
	\end{equation}
	This already looks very much like \eqref{eqn::extended_SH}; however there are some error terms, and moreover, the function $x\mapsto \e^{Cx/2}$ is not bounded. To get around this, we truncate the exponential and control the error. 	
	For fixed $p > 0$, approximating the indicator functions $\I{x\le p}$ by continuous functions \footnote{by say $F_m(x)=\e^{\frac{Cx}{2}}(\I{x\leq p-2^{-m}}+(1-2^m(x-p+2^{-m}))\I{p-2^{-m}\le x \le p})$}, it follows from  \eqref{eqn::extended_SH} that 
	\[\frac{\sqrt{n}}{D_n(\D)} \int_{\D} \e^{2n-2l(z,n)} \e^{\frac{C}{2}\frac{S_n(z)}{\sqrt{n}}} \I{\frac{S_n(z)}{\sqrt{n}}\leq p} \, dz\]
	converges in probability as $\gamma\to 2^{-}$ (and therefore $n\to \infty$) to \[\sqrt{\frac{4}{\pi}} \E{\e^{\frac{C}{\sqrt{2}}R_1}\I{R_1\le p}}.\]
	Since the $o(1)$ is deterministic, the same then also holds for 
	\[\frac{\sqrt{n}}{D_n(\D)} \int_{\D} \e^{2n-2l(z,n)} \e^{\frac{C}{2}(1+o(1))\frac{S_n(z)}{\sqrt{n}}} \I{\frac{S_n(z)}{\sqrt{n}}\leq p} \, dz.\]
	But now for any fixed $C > 0$, we have \[\lim_{p\to \infty} \E{e^{\frac{C}{\sqrt{2}}R_1}\I{R_1\leq p}} = \E{e^{\frac{C}{\sqrt{2}}R_1}}\]
	and one can verify by hand that $\E{e^{mR_1}}\sim \sqrt{2\pi}m e^{m^2/2}$ as $m\to \infty$ (see for example \cite[above equation (4.6)]{Madaule}). Therefore, since we know from Theorem \ref{thm::critical_cascades} that $D_n(\D)\to \mu'(\D)$ almost surely as $n\to \infty$, we can conclude that
	\[D_n(\D) \times \frac{e^{-\frac{C^2}{4}}}{C}\frac{\sqrt{n}}{D_n(\D)} \int_{\D} \e^{2n-2l(z,n)} \e^{\frac{C}{2}(1+o(1))\frac{S_n(z)}{\sqrt{n}}} \I{\frac{S_n(z)}{\sqrt{n}}\leq p} \, dz \]
	converges to $2\mu'(\D)$ in probability, as $n\to \infty$ and then $p\to \infty$. 
	
	Thus, it remains to show that for fixed $C$, 
	\begin{equation}\label{eqn::trunc_p}
	\sqrt{n} \int_{\D} \e^{2n-2l(z,n)} \e^{\frac{C}{2}\frac{S_n(z)}{\sqrt{n}}(1+o(1))} \I{\frac{S_n}{\sqrt{n}}>p} \, dz
	\end{equation}
	converges to $0$ in probability as $\gamma\to 2^{-}$ and then $p\to \infty$. Fix $\epsilon>0$, and recall the definition of the event $E_\eta(n,z)=\{S_k(z)\ge -2\eta\, , \, 0\le k \le n\}$. Let $C_\eta= \cap_{n,z} E_\eta(n,z)$. We now bound 
	\begin{equation}\label{eqn::trunc_p_withepsilon}
	\P\left( \sqrt{n} \int_{\D} \e^{2n-2l(z,n)}\e^{\frac{C}{2}\frac{S_n(z)}{\sqrt{n}}(1+o(1))} \I{\frac{S_n}{\sqrt{n}}>p} \, dz>\epsilon\right)
	\end{equation}
	by the sum of $\P\left( C_\eta ^c \right)$ and 
	\[\P\left(C_\eta \cap \left\{ \sqrt{n} \int_{\D} \e^{2n-2l(z,n)}\e^{\frac{C}{2}\frac{S_n(z)}{\sqrt{n}}(1+o(1))} \I{\frac{S_n}{\sqrt{n}}>p} \, dz >\epsilon \right\}\right).\]
	By the Markov inequality and the fact that $C_\eta \subset \cap_z E_\eta(n,z)$,  the second term is less than
	\begin{equation}
	\label{eqn::truncation_probs}
	\frac{\sqrt{n}}{\epsilon} \E{\int_{\D} \e^{2n-2l(z,n)}\e^{\frac{C}{2}\frac{S_n(z)}{\sqrt{n}}(1+o(1))} \I{\frac{S_n}{\sqrt{n}}>p}\1_{E_\eta(n,z)} \, dz }. 
	\end{equation} 
	Moreover, by definition of the law $\phs$, we see that the expectation in \eqref{eqn::truncation_probs} is equal to a deterministic constant times
	\[ 
	\hat{ \mathbb E}^*\left( e^{\frac{C}{2}\frac{S_n(Z)}{\sqrt{n}}(1+o(1))} \I{\frac{S_n(Z)}{\sqrt{n}}>p} \1_{E_\eta(n,Z)} \right),\]
	which by Lemma \ref{lem::control_exp_big} is less than or equal to 
	\[
	\hat{ \mathbb E}^*\left(c\frac{1+(2\eta+S_0(Z))^+}{\sqrt{n}} \e^{-\frac{p}{4}}\e^{\frac{CS_0(Z)}{\sqrt{n}}}\right)\]
	for $n$ large enough and for some $c(C)>0$. Using that $Z$ is chosen proportionally to $\CR(Z,\D)^2$ under $\hat{ \mathbb P}^*$ we can deduce that for every $\eta>1$ and $n$  large enough, \eqref{eqn::truncation_probs} is less than a deterministic constant times
	$\eta\e^{-p/4}/\epsilon$
	for every fixed $p$. Thus, we can bound \eqref{eqn::trunc_p_withepsilon} by a deterministic constant times
	\[\P\left( C_\eta ^c \right)+\frac{\eta}{\epsilon}\e^{-p/4}.\]
	By \cite[proof of Proposition 6.4]{APS} $\mathbb{P}(C_\eta)\to 1$ as $\eta \to \infty$ and thus by choosing first $\eta$ large, we can make the first term as small as we wish. Then, uniformly in large $n$ by choosing $p$ large, we can also make the second term arbitrarily small. From here the claim follows.
\end{proofof}
\\

\begin{proofof}{Lemma \ref{lem::main2}} Again we write $n=n(\gamma,C)$, and assume that $\Os=\D$. We now use the decomposition \eqref{decomponsition},
	and further separate each component $D'\in \mathcal{A}_n$ into two parts: the points $z$ around which the area of the disk $B(z,d(z,D'))$ is comparable to $\text{Area}(D')$, and the points where $\text{Area}(D')$ is much bigger. More precisely, define 
	\[A_{z,\gamma}= \{\text{Area}(D_n(z))\leq 2^{1/(p-1)} \CR(z,D_n(z))^2\},\]
	where by $D_n(z)$ we denote the component $D' \in \mathcal{A}_n$ containing $z$ and take $p=1+(2-\gamma)/2$ as in Lemma \ref{lem::pmoment}. The reason for choosing this comparison will be clear from the proof.
	
	We now bound $| (M_n^\gamma(\D)-\mu^\gamma(\D))/(2-\gamma) |$ by the sum of
	\begin{equation}\label{eqn::decomp_proofmain2_line1}
	\frac{\sum_{D'\in \mathcal{A}_n} \int_{D'}  \1_{A_{z,\gamma}^c} \left (\e^{\gamma n - \frac{\gamma^2}{2}l(z,n)}dz + \e^{\gamma n}\tilde{\mu}_{D'}^\gamma(dz)\right ) }{2-\gamma}
	\end{equation} 
	and
	\begin{equation}\label{eqn::decomp_proofmain2_line2}\frac{\sum_{D'\in \mathcal{A}_n} \big|\int_{D'}  \1_{A_{z,\gamma}} (\e^{\gamma n - \frac{\gamma^2}{2}l(z,n)}dz - \e^{\gamma n}\tilde{\mu}_{D'}^\gamma(dz))\big| }{2-\gamma}. 
	\end{equation}
	
	We begin by showing that \eqref{eqn::decomp_proofmain2_line1} converges to 0 in $\mathcal{L}^1$ as $\gamma\to 2^{-}$, for any fixed $C$. Indeed, by first conditioning on $A_n$, we see that the expectation of this term is less than or equal to 
	\[ \frac{2}{2-\gamma}\E{\int_{\D} \e^{\gamma n-\frac{\gamma^2}{2} l(z,n)} \1_{A_{z,\gamma}^c} \, dz} = \frac{2}{2-\gamma} \phs_\gamma(A_{Z,\gamma}^c) 
	\]
	which by Lemma \ref{lem::ZbigCR} is bounded above by \[\frac{2}{2-\gamma}\, \big(ce^{-c'\lfloor (\frac{C}{2-\gamma})^2 \rfloor}+ 2^{-\frac{c'}{2(2-\gamma)}}\big)\] for some $c,c'>0$. This nicely converges to $0$ as $\gamma\to 2^{-}$. 
	
	Now, we deal with \eqref{eqn::decomp_proofmain2_line2}. The idea is to use the scaling of $p$-th moments, with $p=1+(2-\gamma)/2 > 1$ as before. To do this denote the whole expression \eqref{eqn::decomp_proofmain2_line2} by $Y_{\gamma,C}$. Fix $\eps>0$. Then for any $\delta>0$ we can write 
	\[ \P(Y_{\gamma,C}>\eps) \leq \P(\, \E{Y_{\gamma,C}^p\mid A_n} > \eps^{p}\delta) + \P(\{ Y_{\gamma,C}^p > \eps^p\} \cap \{\E{Y_{\gamma,C}^p|A_n} \le \eps^{p}\delta\}) \] 
	where by the Markov inequality, the second term is less than $\delta$. Thus, since we can take $\delta$ arbitrarily small, it is sufficient to prove that $\limsup_{C\to \infty} \limsup_{\gamma\to 2^{-}}\P\left(|\E{Y_{\gamma,C}^p\mid A_n}|\ge \epsilon\right)=0$ for any $\epsilon>0$. 
	
	A nice idea from \cite{Madaule} is to now apply the following classical inequality from \cite{EVB}, saying that for any sequence $(X_i)_{i\in \N}$ of independent centered random variables and any $q\in [1,2]$:
	\[ \E{|\sum X_i |^q}\le 2^q \sum \E{ |X_i|^q}.\]
	Applying this to the conditional probability $\E{\cdot \mid A_n}$ and with $q=p$, we see that $\E{Y_{\gamma,c}^p\mid A_n}$ is less than or equal to 
	\[ \left(\frac{2}{2-\gamma}\right)^p \sum_{D\in \D\setminus A_n} \e^{\gamma n p} \left( \left (\int_D \CR(z,D)^{\gamma^2/2}\1_{A_{z,\gamma}} \, dz\right )^p + \E{\left (\int_{D}  \1_{A_{z,\gamma}} \, \tilde{\mu}_D^\gamma(dz)\right )^p| \mathcal{F}_{A_n}} \right). \]
	By using the deterministic inequality $\CR(z,D)^2 \leq 10\text{Area}(D)$ we have
	\[\left(\int_D \CR(z,D)^{\gamma^2/2}\1_{A_{z,\gamma}} \, dz\right )^p \leq 10 \text{Area}(D)^{(1+\gamma^2/4)(p-1)}\int_D \CR(z,D)^{\gamma^2/2}\1_{A_{z,\gamma}} \, dz,\]
	and by Lemma \ref{lem::pmoment} we can bound $\E{\left (\int_{D}  \1_{A_{z,\gamma}} \, \tilde{\mu}_D^\gamma(dz)\right )^p| \mathcal{F}_{A_n}}$ by 
	\[K \text{Area}(D)^{(p-1)(1+\frac{\gamma^2}{4})}\int_D \e^{\gamma n p}\CR(z,D)^{\frac{\gamma^2}{2}} \1_{A_{z,\gamma}}\, dz.\]
	Thus the whole expression by some universal constant times 
	\[ \left(\frac{2}{2-\gamma}\right)^p \sum_{D\in \D\setminus A_n}  \text{Area}(D)^{(p-1)(1+\gamma^2/4)}\int_D \e^{\gamma n p}\CR(z,D)^{\frac{\gamma^2}{2}} \1_{A_{z,\gamma}}\, dz, \]
	which, since $\1_{A_{z,\gamma}} \frac{\text{Area}(D)^{p-1}}{\CR(z,D)^{2p-2}}\le 2$, is in turn less than four times
	\[ \left(\frac{2}{2-\gamma}\right)^p \int_{\D} \e^{\gamma np - (\frac{\gamma^2p}{2} + 2p-2)l(z,n)} \, dz.\] 
	Now we choose $\tilde{\gamma}(\gamma)$ such that $\frac{\tilde{\gamma}^2}{2}=\frac{\gamma^2p}{2} + 2p -2$. A direct calculation yields that 
	\begin{itemize}
		\item $\gamma np-\tilde{\gamma}n=- C^2(1+o(1))/ 4$; 
		\item $(2-\tilde{\gamma})/ (2-\gamma)=:e(\gamma)\to 1$ as $\gamma\to 2^{-}$; and 
		\item $(2-\gamma)^{p-1}\to 1$ as $\gamma\to 2^{-}$.
	\end{itemize}
	Thus, 
	\begin{align*}
	&\lim_{C\to \infty}\limsup_{\gamma\to 2^{-}} \left(\frac{2}{2-\gamma}\right)^p \int_{\D} \e^{\gamma np - (\frac{\gamma^2}{2}p + 2p-2)l(z,n)} \, dz \\
	&\hspace{0.1\textwidth}= \lim_{C\to \infty} \limsup_{\gamma\to 2^{-}} \frac{2e(\gamma)}{2-\tilde{\gamma}} M_{\left \lfloor \frac{C^2 e(\gamma)^2}{(2-\tilde{\gamma})^2}\right \rfloor}^{\tilde{\gamma}}(\D) \e^{-\frac{1}{4}C^2}
	\end{align*}
	where the limits are in probability. Thanks to Lemma \ref{lem::main1}, this is bounded by 
	\[4\mu'(D)\lim_{C\to \infty} e^{-1/4C^2}/C=0.\] 
\end{proofof}

\subsection{Extensions}\label{sec::extensions}

In this section, we will shortly discuss how our results can be extended to the boundary Liouville measure and to the case of the Liouville measure for the Neumann GFF. Our results can be also easily extended to the case of quantum surfaces like quantum wedges, quantum disks or quantum spheres introduced in \cite{ShZ, DMS}, but this will be discussed elsewhere for the brevity of this note \cite{AP}. Given our aim of leaving this a short note, we will not define any of the terms in detail, but rather refer to \cite{Ber2} for the Neumann GFF and to \cite{DS}, for the boundary Liouville measure.

\subsubsection{Liouville measure for the Neumann GFF}\label{sec::NGFF}

We refer the reader to \cite{Ber2} for a definition and disussion on the Neumann GFF. The adaption to (any version of the) Neumann GFF follows by writing the Neumann GFF as a sum of an Dirichlet GFF and the harmonic extension $h$ of an independent log-correlated Gaussian field on the boundary. Notice that this harmonic extension is defined pointwise in the interior of the domain. Whereas the harmonic extension blows up when reaching the boundary, one can check that $m_\gamma(dz) := e^{\gamma h(z)} dz$ still defines an a.s. finite measure on the domain for all $\gamma \in [0,2]$ \cite{HRVdisk}. In particular, as $m_\gamma(dz) \to m_2(dz)$ as $\gamma \to 2$, the case of the Neumann GFF follows from the case of the Dirichlet GFF, when the chaos measures are defined with a different choice of base measure.

\subsubsection{Boundary Liouville measure}

We refer to \cite{DS}, for discussion of the boundary Liouville measure. The most important application is the extension of the Fyodorov-Bouchaud formula \cite{Gui} to the critical case. So for clarity, let us see how our results can be extended to this particular case, where the underlying Gaussian field is defined on the unit circle, with covariance $-2\log ||x-y||_2$. In fact, it is easier to generalize our argument first to the case of boundary measures associated with a  Neumann-Dirichlet GFF (on the ``Neumann'' part of the boundary), and then to conclude the result for the circular boundary measure above, by absolute continuity. 

So let us discuss the case of the Neumann-Dirichlet GFF. It was already explained in \cite{APS}, Section 5, how to extend our construction of the Liouville measure using FPS to this boundary measure. However, this was only done in the subcritical regime.

The first step in adapting the proof therefore, is to provide a construction of the critical boundary measure using the boundary equivalent of the first passage sets. These boundary-FPS are discussed in Section 5 of \cite{APS} and their behaviour is completely analoguous to the normal FPS. In fact, via the boundary-FPS, the proofs in Section 6 of \cite{APS} will work essentially word-for-word to prove that one can construct a critical boundary measure using the derivative martingale and using a Seneta-Heyde scaling, and that these constructions agree (up to explicit constants) with the critical boundary measure as constructed using semi-circle averages of the field \footnote{in the Seneta--Heyde scaling, this is \cite[Theorem 4.1]{HRVdisk}}. One only needs to replace the relevant definitions for sets, conformal radius etc, exactly as done in Section 5 of \cite{APS} for the subcritical case. For clarity, we also list here the external inputs to Section 6, and how they extend to the critical case:
\begin{itemize}
	\item \emph{Lemmas 2.3 and 3.5 from the article \cite{Ai}.} One can check that these also hold for the boundary-loops; the arguments are based on the iterative nature of conformal loop ensembles and their conformal invariance, both of which hold for the boundary loop ensembles.
	\item \emph{Theorem 1.1 from the article \cite{EP}}, which says that the critical Liouville measure for a Dirichlet GFF in the bulk (\cite{DKRV,JS}) can equivalently be constructed using the ``derivative martingale'' defined via circle averages of the field. The proof in \cite{EP} directly adapts to the setting of boundary measures for the Neumann--Dirichlet GFF. Indeed, the argument is based around certain changes of measure for the Brownian motions arising from circle averages of the Dirichlet GFF in the bulk, and when one instead considers semi-circle averages of the Neumann--Dirichlet GFF on the boundary, these remain Brownian motions. The only change is that they have speed $2$. Thus, one obtains that the boundary derivative martingale defined using semi-circle averages of the Neumann--Dirichlet GFF gives an equivalent construction of the critical boundary measure defined in \cite[Theorem 4.1]{HRVdisk}. 
	\item \emph{Proposition 3.6 from \cite{EP}}. This states that certain cut-off versions of the (bulk) derivative martingale are uniformly integrable. For the same reason as above the proof extends directly to give the equivalent result for the boundary derivative martingale. 
\end{itemize}

The second, and final, step is to adapt the proof of the current article to the Dirichlet--Neumann case. Again, this goes through word-for-word when one replaces the relevant definitions appropriately.

\subsection* {Acknowledgements} Juhan Aru is supported by the SNF grant 175505, Ellen Powell is supported in part by the NCCR Swissmap and Avelio Sepúlveda is supported by the ERC grant LiKo 676999. All the authors are also thankful to the SNF grant 155922 and the NCCR Swissmap initiative.

\bibliographystyle{alpha}
\bibliography{bibliography_gmc_cascades}

\end{document}